\documentclass[arxiv]{agn_article}
\PassOptionsToPackage{sortcites}{agn_bib}

\usepackage[
english]{babel}

\usepackage{amssymb,amsmath,amsthm,agn_bib}

\usepackage{agn_macros,agn_authors}
\usepackage{mathrsfs}

\AtBeginBibliography{\footnotesize}

\newtheorem{theorem}{Theorem}[section]

\newtheorem{lemma}[theorem]{Lemma}
\newtheorem{observation}[theorem]{Observation}
\newtheorem{corollary}[theorem]{Corollary}

\theoremstyle{definition}

\newtheorem{remark}[theorem]{Remark}

\usepackage{xcolor,hyperref}
\definecolor{darkblue}{rgb}{0,0,0.6}
\hypersetup{
    colorlinks=true,       
    linkcolor=darkblue,          
    citecolor=darkblue,        
    filecolor=darkblue,      
    urlcolor=darkblue           
}

\usepackage[colorinlistoftodos,prependcaption,textsize=tiny]{todonotes}
\newcommand{\skalarProd}[2]{\big\langle#1,#2\big\rangle}

\renewcommand{\D}{\operatorname{D}\hspace{-1pt}}
\newcommand{\Anti}{\operatorname{Anti}}

\usepackage{tikz}

\begin{document}

\begin{tikzpicture}[remember picture, overlay]
 \node [xshift=-1cm,yshift=15cm,rotate=-90] at (current page.south east)
 {The final publication appeared in Mathematical Methods in the Applied Sciences (2021), doi: \href{http://doi.org/10.1002/mma.7498}{10.1002/mma.7498}.
 };
\end{tikzpicture}

\title{Ne\v{c}as-Lions lemma revisited: An $L^p$-version of the generalized Korn inequality for incompatible tensor fields}
\knownauthors[lewintan]{lewintan, neff}

\maketitle

\begin{abstract}
For $1<p<\infty$ we prove an $L^p$-version of the generalized Korn  inequality for incompatible tensor fields $P$ in $ W^{1,\,p}_0(\Curl; \Omega,\R^{3\times3})$. More precisely, let $\Omega\subset\R^3$ be a bounded Lipschitz domain. Then there exists a constant $c=c(p,\Omega)>0$ such that
\begin{equation*}
 \norm{ P }_{L^p(\Omega,\R^{3\times3})}\leq c\,\left(\norm{ \sym P }_{L^p(\Omega,\R^{3\times3})} + \norm{ \Curl P }_{L^p(\Omega,\R^{3\times3})}\right)
\end{equation*}
holds for all tensor fields ~ $P\in  W^{1,\,p}_0(\Curl; \Omega,\R^{3\times3})$,~  i.e., for all $P\in W^{1,\,p}(\Curl; \Omega,\R^{3\times3})$ with vanishing tangential trace $ P\times \nu=0 $ on $ \partial\Omega$
where  $\nu$ denotes the outward unit normal vector field to $\partial\Omega$.\\
For compatible $P=\D  u$ this recovers an $L^p$-version of the classical Korn's first inequality
and for skew-symmetric $P=A$ an $L^p$-version of the Poincar\'{e} inequality.
\end{abstract}



\section{Introduction}

In this paper we generalize the main result from \cite{agn_neff2015poincare} for $n=3$ to the $L^p$-setting. This is, we prove
\begin{equation}\label{eq:Korn_Lp}
 \norm{ P }_{L^p(\Omega,\R^{3\times3})}\leq c(p,\Omega)\,\left(\norm{ \sym P }_{L^p(\Omega,\R^{3\times3})} + \norm{ \Curl P }_{L^p(\Omega,\R^{3\times3})}\right) \qquad \forall P\in  W^{1,\,p}_0(\Curl; \Omega,\R^{3\times3}),
\end{equation}
where, in the classical sense, the vanishing tangential trace reads \ $ P\times \nu=0 $ \ on $ \partial\Omega$
and  $\nu$ denotes the outward unit normal vector field on the boundary of a bounded Lipschitz domain $\Omega$.

The original proof of \eqref{eq:Korn_Lp} in $L^2$, cf.~\cite{agn_neff2015poincare}, is rather technical and uses the classical Korn's inequality, the Maxwell-compactness property and suitable Helmholtz-decompositions of tensor fields together with a restricting assumption on the domain to be ``slicable'' and is, moreover, not directly amenable to the $L^p$-case. Our new argument essentially uses only the Lions lemma resp. Ne\v{c}as estimate (Theorem \ref{th:necas_estimate}), the compactness of $W^{1,\,p}_0(\Omega)\subset\!\subset L^p(\Omega)$ and the algebraic identity
\begin{equation}\label{eq:prod_id}
(\Anti a)\times b = b \otimes a -\skalarProd{b}{a}\, \id \qquad \forall\,  a,b\in\R^3,
\end{equation}
where $\Anti:\R^3\to\so(3)$ is the canonical identification of vectors with skew-symmetric matrices, see Section \ref{sec:Notation}. For a Lipschitz domain (i.e. open connected with Lipschitz boundary) $\Omega\subset\R^n$, the Lions lemma states that $f\in L^p(\Omega)$ if and only if $f\in W^{-1,\,p}(\Omega)$ and $\nabla f \in W^{-1,\,p}(\Omega,\R^n)$, which is equivalently expressed by the Ne\v{c}as estimate
\begin{equation}
 \norm{f}_{L^p(\Omega)}\le c\, (\norm{f}_{ W^{-1,\,p}(\Omega)}+\norm{\nabla f}_{ W^{-1,\,p}(\Omega,\R^n)})
\end{equation}
with a positive constant $c=c(p,n,\Omega)$.

Let $\Omega\subset\R^n$ be a bounded Lipschitz domain. In this text, we refer to Korn's first inequality (in $L^p$) with vanishing boundary values\footnote{In fact, the estimate is true for functions with vanishing boundary values on a relatively open (non-empty) subset of the boundary.}  as the statement
\begin{equation}\label{eq:Korn1}
 \exists\, c > 0 \ \forall\, u\in W^{1,\,p}_{0}(\Omega,\R^n) :  \qquad \norm{\D  u}_{L^p(\Omega,\R^{n\times n})}\leq c\, \norm{\sym \D  u}_{L^{p}(\Omega,\R^{n\times n})}.
\end{equation}
It can be obtained from Korn's second inequality (in $L^p$), which does not require boundary conditions and which reads
\begin{equation}\label{eq:Korn2}
  \exists\, c > 0 \ \forall\, u\in W^{1,\,p}(\Omega,\R^n) :  \qquad \norm{u}_{W^{1,\,p}(\Omega,\R^n)}\leq c\, \left(\norm{u}_{L^p(\Omega,\R^n)} + \norm{\sym \D  u}_{L^{p}(\Omega,\R^{n\times n})} \right).
\end{equation}
A further consequence of the latter inequality is the quantitative version
\begin{equation}\label{eq:KornQuant}
 \exists\, c > 0 \ \forall\, u\in W^{1,\,p}(\Omega,\R^n)\ \exists\, \widetilde{A}\in\so(n) :  \qquad \norm{\D  u- \widetilde{A}}_{L^p(\Omega,\R^{n\times n})}\leq c\, \norm{\sym \D  u}_{L^{p}(\Omega,\R^{n\times n})}.
\end{equation}
These inequalities are crucial for a priori estimates in linear elasticity and fluid mechanics and hence cornerstones for well-posedness results in linear elasticity ($L^2$-setting) and the Stokes-problem ($L^p$-setting). They
can be generalized to many different settings, including the geometrically nonlinear counterpart
\cite{friesecke2002rigidity,LM2016optimalconstants}, mixed growth conditions \cite{CDM2014mixedgrowth}, incompatible  fields (also with dislocations) \cite{MSZ2014incompatible,agn_neff2015poincare,agn_neff2012maxwell,agn_bauer2013dev} and trace-free infinitesimal strain measures \cite{Dain2006tracefree,agn_jeong2008existence,Reshetnyak1970,Schirra2012tracefreenD}. Other generalizations are applicable to Orlicz-spaces \cite{Fuchs2010,BD2012Orlicz,BCD2017Orlicz,Cianchi2014Orlicz} and SBD functions with small jump sets \cite{Friedrich2017SBD,CCF2016SBD,Friedrich2018SBD,KohnTemam1983SBD}, thin domains \cite{LM2011thindomians,GH2018thindomains,Harutyunyan2017thindomains} and John domains \cite{DM2004Jonesdomains,DRS2010Johndomains,ADM2006Johndomains} as well as the case of non-constant coefficients \cite{agn_neff2002korn,agn_lankeit2013uniqueness,agn_neff2014counterexamples,Pompe2003Korn}. Piecewise Korn-type inequalities subordinate to a FEM-mesh and involving jumps across element boundaries have also been investigated, see e.g.\ \cite{Brenner2004Korn,agn_lew2004optimal}. Korn's inequalities fail for $p=1$ and $p=\infty$, see \cite{CFM2005counterex,dlM1964counterex,Ornstein1962}.

There exist many different proofs of the classical Korn's inequalities, see the discussions in \cite{Ciarlet2010,agn_neff2015poincare,ACM2015,KO88,Nitsche81,DuvautLions72,Geymonat86} as well as \cite[Sect. 6.15]{Ciarlet2013FAbook} and the references contained therein. A rather concise and elegant argument uses the well-known representation of the second distributional derivatives of  the displacement $u$ by a linear combination of the first derivatives of the symmetrized gradient:
\begin{subequations}\label{eq:sec_der_id}
\begin{equation}
 \partial_i\partial_j u_k = \partial_j(\sym \D  u)_{ik}+ \partial_i(\sym \D  u)_{jk}-\partial_k(\sym \D  u)_{ij},
\end{equation}
i.e.
\begin{equation}
\D\,  (\D  u) = L(\D  \sym \D  u), \quad \text{with a linear operator } L.
\end{equation}
\end{subequations}
Then $\sym \D  u=0$ implies that $u$ is a first order polynomial. Furthermore, for $1<p<\infty$, the Lions lemma resp.\ Ne\v{c}as estimate (Theorem \ref{th:necas_estimate}) applied to \eqref{eq:sec_der_id} yields a variant of Korn's second inequality in $W^{1,\,p}(\Omega)$ from which, in turn, the first Korn's inequality (with boundary conditions) can be deduced \cite{Ciarlet2010, Geymonat86, DuvautLions72,ACM2015} using an indirect argument together with the compactness of the dual spaces
\begin{equation}\label{eq:dualcompact}
W^{1,\,p}_0(\Omega)\subset\!\subset L^p(\Omega)=  \left(L^{p'}(\Omega)\right)'\subset\!\subset \left(W^{1,\,p'}_0(\Omega)\right)'=W^{-1,\,p}(\Omega)
; \qquad
\frac1p+\frac{1}{p'}=1.
\end{equation}
Furthermore, such an argumentation scheme also applies to obtain Korn-type inequalities on surfaces \cite{Duduchava2010,CHM2016,Ciarlet98Kornsurfaces}, in Sobolev spaces with negative exponents \cite{ACC2010KorninDual} and in weighted homogeneous Sobolev spaces \cite{DRS2010Johndomains}.

Korn's inequalities can be generalized to incompatible square tensor fields $P$ if one adds a term  in $\Curl P$
on the right hand side, cf. \cite{agn_neff2015poincare}, thus extending Korn's first inequality to
incompatible tensor fields having vanishing restricted tangential trace on (a relatively open subset of) the boundary. For recent refined estimates which involve only the deviatoric (i.e. trace free) part of
 $\sym P$ and $\Curl P$, see \cite{agn_bauer2013dev}.  In the two-dimensional case, an even
stronger estimate holds true for fields $P\in L^1(\Omega,\R^{2\times2})$ with $\Curl P\in L^1(\Omega,\R^2)$; then $P\in L^2(\Omega,\R^{2\times2})$ and
\begin{equation}
\norm{P}_{L^2(\Omega,\R^{2\times2})} \leq c\,\left(\norm{\sym P}_{L^2(\Omega,\R^{2\times2})} + \norm{\Curl P}_{L^1(\Omega,\R^2)}\right)
\end{equation}
under the normalization condition $\int_\Omega \skew P\, \intd{x}= 0$, cf. \cite{Garroni10}. However, for applications, it is preferable to work in the three-dimensional case and under more natural tangential boundary conditions.

Our new inequality \eqref{eq:Korn_Lp} is originally motivated from infinitesimal gradient plasticity with plastic spin. There, one introduces the additive decomposition
$$
\D  u = e + P
$$
of the displacement gradient $\D  u\in\R^{3\times3}$ into incompatible non-symmetric elastic distortion $e$ and incompatible plastic  distortion $P$. Then the thermodynamic potential generically has the form
\begin{equation}\label{eq:plast_potential}
\begin{split}
\int_\Omega\frac12\norm{\sym e}^2 +& \frac12\norm{\sym P}^2+ \frac12\norm{\Curl e}^2 -\skalarProd{f}{u}\,\intd{x}\\
&= \int_\Omega \frac12\norm{\sym (\D  u - P)}^2+\frac12\norm{\sym P}^2 + \frac12\norm{\Curl P}^2 -\skalarProd{f}{u}\,\intd{x},
 \end{split}
\end{equation}
where $f\in L^2(\Omega,\R^3)$ describes the body force, see e.g. \cite{agn_ebobisse2016canonical, agn_ebobisse2018well,RS2017viscoplast,agn_neff2009notes}. Here, $\norm{\sym e}^2$ represents the elastic energy,  $\norm{\sym P}^2$ induces linear hardening and $\Curl P$ is known as the \emph{dislocation density tensor}.  The $L^2$-generalized Korn's inequality establishes coercivity of \eqref{eq:plast_potential} with respect to displacements and plastic distortions, for example, if Dirichlet boundary condition $u_{|_{\Gamma_D}}=0$ and consistent tangential boundary conditions $P\times \nu\,_{|_{\Gamma_D}}=0$ are prescribed. Crucial for plasticity theories with spin is that the plastic contribution cannot be reduced to a dependence on the symmetric plastic strain $\varepsilon_p\coloneqq\sym P$ alone, as can be done in classical plasticity. The system of  equations connected to \eqref{eq:plast_potential} reads
\begin{equation}\label{eq:thermo}
 \begin{split}
  \Div (\sym(\D  u -P)) = f, \qquad \text{(balance of forces)}\\
  \dot{P}\in\partial I_K(\sym(\D  u -P)-\sym P - \Curl \Curl P), \qquad \text{(plastic flow law)}
 \end{split}
\end{equation}
together with appropriate initial and boundary conditions, where $I_K$ is the indicator function of a convex domain $K$. The $L^p$-version presented in this article may then serve to show well-posedness results for nonlinear dislocation mediated hardening. Notably, analytic examples suggest to use $\norm{\Curl P}^q$ with $1<q<2$ for the nonlocal dislocation backstress, cf. \cite{CO2005dislocation,Garroni10}.

Another field of application of the generalized Korn's inequality for incompatible tensor fields is the so-called \emph{relaxed micromorphic model}, see \cite{agn_neff2015relaxed,agn_madeo2016reflection,agn_madeo2018modeling}. In this generalized continuum model, the task is to find the macroscopic displacement $u:\Omega\subset \R^3 \to \R^3$ and the (still macroscopic) micro-distortion tensor $P:\Omega\subset\R^3\to \R^{3\times3}$ minimizing the elastic energy
\begin{equation}\label{eq:energy}
\begin{split}
 &\int_{\Omega} \frac12\norm{\sym (\D  u - P)}^2+\frac12\norm{\sym P}^2 + \frac12\norm{\Curl P}^2  -\skalarProd{f}{u}\,\intd{x}\to \min\,,\\
 &\hspace{5em}(u,P)\in H^1(\Omega,\R^3)\times H(\Curl;\Omega,\R^{3\times3}) \quad\text{ with } \ u_{|_{\Gamma_D}}= g, \ P\times \nu\,_{|_{\Gamma_D}} = \D  u\times \nu\,_{|_{\Gamma_D}}.
\end{split}
\end{equation}
The equilibrium equations are the Euler-Lagrange equations to \eqref{eq:energy}, which read
\begin{equation}\label{eq:EL}
\begin{split}
 \Div \sym (\D  u -P) &= f  \hspace{7.4em}\text{(balance of forces)},\\
 \Curl \Curl P + \sym P &= \sym(\D  u -P) \qquad \text{(generalized balance of angular momentum)}.
\end{split}
 \end{equation}
Here, \eqref{eq:EL}$_2$ represents a \emph{tensorial Maxwell problem} in which, due to the appearance of $\sym P$ (instead of $P$), the equations are strongly coupled. Note that the appearance of $\sym P$ in \eqref{eq:plast_potential} and \eqref{eq:energy} is dictated by invariance of the model under infinitesimal rigid body motions. The well-posedness of the weak formulation depends on Korn-type inequalities for incompatible tensor fields, see \cite{agn_bauer2013dev,agn_neff2015poincare}. Dynamic versions of this model allow for the description of frequency band-gaps as observed in metamaterials, cf.  \cite{agn_madeo2016reflection}. The band-gap property crucially depends on using $\Curl P$ in the model.

Let us mention as third application the $p$-$\Curl\Curl$ problem \cite{LL2019pcurlcurl,laforest2018pcurlcurl,MRS2012pcurlcurl} appearing in modeling the magnetic field in a high-temperature superconductor. By the definition of the Banach space $W^{1,\,p}(\Curl;\Omega,\R^{3\times3})$ it is clear that
\begin{equation}\label{eq:curlcurl}
 \Curl(\norm{\Curl P}^{p-2}\Curl P) + P = G, \quad P\times \nu _{|\partial \Omega} =0
\end{equation}
with $G\in L^{p'}(\Omega,\R^{3\times3})$ admits a unique solution for $1<p\leq2$, since \eqref{eq:curlcurl} is the Euler-Lagrange equation to the strictly convex minimization problem
\begin{equation}
 \int_{\Omega} \frac12\norm{P}^2+\frac1p \norm{\Curl P}^p-\skalarProd{G}{P}\,\intd{x} \to \min.
\end{equation}
Our new a-priori estimate \eqref{eq:Korn_Lp} then allows us to show existence and uniqueness for $1<p\leq 2$ of weak solutions $P\in W^{1,\,p}(\Curl;\Omega,\R^{3\times3})$ to
\begin{equation}
 \Curl(\norm{\Curl P}^{p-2}\Curl P) + \sym P = G, \quad P\times \nu _{|\partial \Omega} =0
\end{equation}
resp.\
\begin{equation}
 \int_{\Omega} \frac12\norm{\sym P}^2+\frac1p \norm{\Curl P}^p-\skalarProd{G}{P}\,\intd{x} \to \min.
\end{equation}

\section{Notation and preliminaries}\label{sec:Notation}
We denote by   $\skalarProd{.}{.}$ the scalar product and by $.\otimes .$ the dyadic product. In $\R^3 $ we moreover make use of the cross product $.\times.$~. Since for a fixed vector $a\in\R^3$ the vector product $a\times .$  is linear in the second component there exists a unique matrix $\Anti(a)$ with the property
\begin{equation}
 a\times b \eqqcolon \Anti(a)\,b \qquad \forall \ b\in\R^3.
\end{equation}
For $a=(a_1,a_2,a_3)^T$ the matrix $\Anti(a)$ has the form
\begin{equation}
 \Anti(a)=\begin{pmatrix}
           0 & -a_3 & a_2 \\ a_3 & 0 & -a_1 \\ -a_2 & a_1 & 0
          \end{pmatrix}\,,
\end{equation}
so that $\Anti:\R^3\to\so(3)$ identifies $\R^3$ with the space of skew-symmetric matrices $\so(3)$ canonically, and allows for a generalization towards a vector product of a matrix $P\in\R^{3\times3}$ and a vector $b\in\R^3$ via
\begin{equation}
 P\times b \coloneqq P \Anti(b)
\end{equation}
which is seen as taking the row-wise cross product with $b$. Of crucial importance in our considerations is the relation
 \begin{align}\label{eq:prod_id_ausgeschrieben}
  (\Anti (a))\times b =  \begin{pmatrix}
                        -a_2b_2-a_3b_3 & a_2b_1 & a_3b_1 \\
                        a_1b_2 & -a_1b_1-a_3b_3 & a_3b_2 \\
                        a_1b_3 & a_2b_3 & -a_1b_1-a_2b_2
                       \end{pmatrix}
                     = b\otimes a - \skalarProd{b}{a}\,\id\,.
 \end{align}
An easy consequence is
\begin{observation}\label{obs:alg_statement}
 For $a,b\in\R^3$ we have
 \begin{equation}\label{eq:norm_abschaetzung}
 \norm{a}\norm{b}\le\norm{(\Anti (a))\times b}\le\sqrt{2}\,\norm{a}\norm{b}\,.
 \end{equation}
\end{observation}
\begin{proof}
Taking the squared norm on both sides of \eqref{eq:prod_id} we obtain
\begin{align}
 \norm{(\Anti(a))\times b}^2= \norm{b\otimes a}^2+3\skalarProd{b}{a}^2-2\skalarProd{b}{a}^2 = \norm{a}^2\norm{b}^2+\skalarProd{a}{b}^2
\end{align}
and the right hand side is bounded from below by $\norm{a}^2\norm{b}^2$ and from above by $2\,\norm{a}^2\norm{b}^2$. These bounds are sharp if $a$ is perpendicular to $b$ and if $a$ is parallel to $b$, respectively.
\end{proof}
\begin{remark}
 By the identification of skew-symmetric matrices with vectors in $\R^3$ the estimate \eqref{eq:norm_abschaetzung} also reads
 \begin{equation}
  \frac12\norm{A}\norm{\widetilde{A}}\le\norm{A\,\widetilde{A}}\le\frac{\sqrt{2}}{2}\norm{A}\norm{\widetilde{A}} \qquad \forall A,\widetilde{A}\in\so(3).
 \end{equation}
Indeed, with the choice of $a,\widetilde{a}\in\R^3$ such that $A=\Anti(a)$ and $\widetilde{A}=\Anti(\widetilde{a})$, we have $A\,\widetilde{A}=\Anti(a)\times\widetilde{a}$, so that in regard with $\norm{A}^2=2\norm{a}^2$ and $\norm{\widetilde{A}}^2=2\norm{\widetilde{a}}^2$ the estimate follows.
\end{remark}

\subsection{Basic considerations from Clifford analysis}

The vector differential operator $\nabla$ behaves algebraically like a vector, so that the gradient, the divergence and curl of a vector field $a\in\mathscr{D}'(\Omega,\R^3)$ can be formally viewed as
\begin{equation}\label{eq:vector_diff}
 \D  a = a \otimes \nabla = (\nabla \otimes a)^T, \  \div a = \skalarProd{a}{\nabla}=\skalarProd{\nabla}{a} = \skalarProd{\nabla a}{\id} = \tr{\D  a} \ \text{ and} \ \operatorname{curl} a = a \times (-\nabla) = \nabla \times a\,.
\end{equation}
These operations generalize to $(3\times 3)$-matrix fields row-wise. (In fact, we understand by $\D  P$ the full gradient of $P\in \mathscr{D}'(\Omega,\R^{3\times 3})$.)  So, especially, the matrix $\Curl: \mathscr{D}'(\Omega,\R^{3\times 3})\to\mathscr{D}'(\Omega,\R^{3\times 3})$ is introduced as
\begin{equation}
 \Curl P \coloneqq P \times (-\nabla)= P\,\Anti(-\nabla).
\end{equation}
Comparable with the classical representation of the second derivatives of the displacement by linear combinations of first derivatives of its symmetrized gradient, cf. \eqref{eq:sec_der_id}, the fundamental relation \eqref{eq:prod_id} implies that the full gradient of a skew-symmetric matrix is already determined by its $\Curl$:
\begin{corollary} \label{cor:lin_combi}
 For $A\in \mathscr{D}'(\Omega,\so(3))$ the entries of the derivative $\D  A$ are linear combinations of the entries from $\Curl A$.
\end{corollary}
\begin{proof} Let us consider the distributional version of $\Anti:\mathscr{D}'(\Omega,\R^{3})\to\mathscr{D}'(\Omega,\so(3))$, i.e., for each $a=(a_1,a_2,a_3)^T\in\mathscr{D}'(\Omega,\R^{3})$ we have \[\Anti(a)\coloneqq \begin{pmatrix}0 & -a_3 & a_2 \\ a_3 & 0 & -a_1 \\ -a_2 & a_1 &0 \end{pmatrix}\in\mathscr{D}'(\Omega,\so(3)). \]
Thus, for 
$A=\Anti(a)$, we find
\begin{equation}\label{eq:Curl_skew}
 \Curl A = (\Anti a) \times (-\nabla) \overset{\eqref{eq:prod_id}}{=}
 \tr(\D  a)\,\id- (\D  a)^T,
\end{equation}
thus
\begin{equation}
\D  a = \frac12 (\tr[\Curl A])\id - (\Curl A)^T.
\end{equation}
Therefore, the entries of  $\D  A=\D\Anti(a)$ are linear combinations of the entries from $\Curl A$:
\begin{equation*}
 \D  A = L(\Curl A) \qquad \text{for any skew-symmetric matrix field $A$.}
 \qedhere
\end{equation*}
\end{proof}
\begin{remark}
 Equation \eqref{eq:Curl_skew} is also known as Nye's formula, cf. \cite[eq.\!\! (7)]{Nye53}, but is, in fact, a special case of the algebraic relation \eqref{eq:prod_id}. Interestingly, relation \eqref{eq:Curl_skew} admits also a counterpart on $\SO(3)$ and even in higher spatial dimensions, see \cite{agn_munch2008curl}.
\end{remark}

\begin{corollary}\label{cor:Nye}
 Let $\Omega$ be connected and $A\in L^p(\Omega,\so(3))$. Then we have $\Curl A = 0$ in the distributional sense if and only if $A= \operatorname{const.}$ almost everywhere in $\Omega$.
 \end{corollary}
\begin{proof}
Follows from the expression $\D A = L(\Curl A)$ in the distributional sense.
\end{proof}

\subsection{Function spaces and equivalence of norms}

By definition of the norm in the dual space, it is clear that for $u\in L^p(\Omega,\R^d)$ we have $u\in W^{-1,\,p}(\Omega,\R^d)$ and $\D u\in W^{-1,\,p}(\Omega,\R^{d\times n})$ in the distributional sense, since for all $\varphi\in  W^{1,\,p'}_0(\Omega,\R^d)$, with $\frac1p+\frac{1}{p'}=1$, it holds
\begin{subequations}
\begin{align}
\left| \int_\Omega \skalarProd{u}{\varphi}\,\intd{x} \right|  &\leq \norm{u}_{L^p(\Omega,\R^d)}\,\norm{\varphi}_{W^{1,\,p'}(\Omega,\R^d)}\ , \quad \text{and}\\
\left| \int_\Omega \skalarProd{\partial_i u}{\varphi}\,\intd{x} \right| &= \left| \int_\Omega \skalarProd{u}{\partial_i \varphi}\,\intd{x} \right| \leq \norm{u}_{L^p(\Omega,\R^d)}\,\norm{\varphi}_{W^{1,\,p'}(\Omega,\R^d)}\,.
\end{align}
\end{subequations}
It is remarkable and a deep result that a converse implication holds true as well.
\begin{theorem}[Lions lemma and Ne\v{c}as estimate]  \label{th:necas_estimate}
Let $\Omega\subset\R^n$ be a bounded  Lipschitz  domain.  Let $m \in \Z$ and $p \in (1, \infty)$.
Then $f \in  \mathscr{D}'(\Omega,\R^d)$ and $\D  f \in W^{m-1,\,p}(\Omega,\R^{d\times n})$ imply
$f \in W^{m,\,p}(\Omega,\R^d)$.
Moreover, 
\begin{equation}  \label{eq:necas_m_p}
 \norm{ f}_{W^{m,\,p}(\Omega,\R^d)} \le c\,\left(\norm{ f}_{W^{m-1,\,p}(\Omega,\R^d)} + \norm{ \D  f }_{W^{m-1,\,p}(\Omega,\R^{d\times n})}\right),
\end{equation}
with a constant $c=c(m,p,n,d,\Omega)>0$.
\end{theorem}
For the proof see \cite[Proposition 2.10 and Theorem 2.3]{AG90}, \cite{BS90}.  In our following discussions, the heart of the matter is the estimate  \eqref{eq:necas_m_p}, see  Ne\v{c}as \cite[Th\'{e}or\`{e}me 1]{Ne66}. In fact, the case $m=0$ is already contained in \cite{Cattabriga61}; for an alternative proof, see \cite[Lemma 11.4.1]{MW2012bvpStokes} and \cite[Chapter IV]{BF2013} as well as \cite{Bramble2003}.   For further historical remarks, see the discussions in \cite{Ciarlet2010,ACM2015} and the references contained therein.

\begin{remark}
 Note that in the case $d=n$ it suffices to consider the symmetrized gradient operator \ $\sym \D $ \ instead of the full gradient $\D $, since one can express the second distributional derivatives of a vector field by a linear combination of its first derivatives of the symmetrized gradient, see \eqref{eq:sec_der_id}. Moreover, we deduce the estimate
\begin{equation}
  \norm{ f}_{ W^{m,\,p}(\Omega,\R^n)} \le c\,\left(\norm{ f}_{ W^{m-1,\,p}(\Omega,\R^n)} + \norm{ \sym \D  f }_{ W^{m-1,\,p}(\Omega,\R^{\vphantom{d}n\times n})}\right)
\end{equation}
 which is Korn's second inequality in $L^p$ or $W^{1,\,p}$ with $m=0$ resp. $m=1$, cf. \cite{ACGK2006,CMM2018} for the case $p=2$.
\end{remark}

For the subsequent considerations, we shall focus on the three-dimensional case $n=3$.
We will work in the Banach-space
\begin{subequations}
\begin{align}
  W^{1,\,p}(\Curl; \Omega,\R^{3\times3}) &\coloneqq \{P\in L^p(\Omega,\R^{3\times3})\mid \Curl P \in L^p(\Omega,\R^{3\times3})\}
 \shortintertext{equipped with the norm}
 \norm{P}_{ W^{1,\,p}(\Curl; \Omega,\R^{3\times3})}&\coloneqq \left(\norm{P}^p_{L^p(\Omega,\R^{3\times3})} + \norm{\Curl P}^p_{L^p(\Omega,\R^{3\times3})} \right)^{\frac{1}{p}}.
\end{align}

\end{subequations}
The density of $\mathscr{D}(\Omega,\R^{3\times3})$ in $ W^{1,\,p}(\Curl; \Omega,\R^{3\times3})$ follows by standard arguments. Furthermore, we consider the subspace
\begin{align*}
  W^{1,\,p}_0(\Curl; \Omega,\R^{3\times3}) \coloneqq \{P\in  W^{1,\,p}(\Curl; \Omega,\R^{3\times3}) \mid P \times \nu = 0 \text{ on } \partial \Omega\},
\end{align*}
where $\nu$ denotes the outward unit normal vector field to $\partial\Omega$,
and the tangential trace $P\times \nu$ is understood in the sense of $W^{-\frac1p,\, p}(\partial \Omega,\R^{3\times3})$ which is justified by integration by parts, so that its trace is defined by
\begin{equation}
 \forall\ Q\in  W^{1-\frac{1}{p'},\,p'}(\partial\Omega,\R^{3\times 3}) :  \quad \skalarProd{P\times (-\nu)}{Q}_{\partial \Omega}=  \int_{\Omega}\skalarProd{\Curl P}{\widetilde{Q}}-\skalarProd{P}{\Curl \widetilde{Q}}\, \intd{x},
\end{equation}
where $\widetilde{Q}\in W^{1,\,p'}(\Omega,\R^{3\times3})$ denotes any extension of $Q$ in $\Omega$. Here, $\skalarProd{.}{.}_{\partial\Omega}$ indicates the duality pairing between $W^{-\frac1p,\,p}(\partial\Omega,\R^{3\times3})$ and $W^{1-\frac{1}{p'},\,p'}(\partial\Omega,\R^{3\times 3})$.
\section{Main results}

We shall start with the following

\begin{lemma}\label{lem:basic}
  Let $\Omega \subset \R^3$ be a bounded Lipschitz domain and $1<p<\infty$. Then $P\in\mathscr{D}'(\Omega,\R^{3\times3})$, $\sym P\in L^p(\Omega,\R^{3\times3})$ and $\Curl P \in W^{-1,\,p}(\Omega,\R^{3\times3})$ imply $P\in L^p(\Omega,\R^{3\times3})$. Moreover, we have the estimate
  \begin{equation}\label{eq:basic}
   \norm{P}_{L^p(\Omega,\R^{3\times3})} \leq c\, \left(\norm{\skew P}_{W^{-1,\,p}(\Omega,\R^{3\times3})}+\norm{\sym P}_{L^p(\Omega,\R^{3\times3})}+ \norm{ \Curl P }_{W^{-1,\,p}(\Omega,\R^{3\times3})}\right),
  \end{equation}
  with a constant $c=c(p,\Omega)>0$.
\end{lemma}

\begin{remark}
In the proof of our generalized Korn-type inequalities (Theorem \ref{thm:main1} and Theorem \ref{thm:main2}) we make use of the compact embedding $L^p(\Omega)\subset\!\subset W^{-1,\,p}(\Omega)$, see \eqref{eq:dualcompact}, so that the $W^{-1,\,p}$-norm of the first term on the right hand side is of crucial importance and will not be estimated by the $L^p$-norm.
\end{remark}

\begin{proof}[Proof of Lemma \ref{lem:basic}]
It suffices to deduce that $\skew P\in L^p(\Omega,\R^{3\times 3})$ from the assumption of the lemma. By the linearity of the operator $\Curl$ and the decomposition $P=\skew P + \sym P$ holding in $\mathscr{D}'(\Omega,\R^{3\times3})$, the following subsists \[\Curl\skew P =\Curl P -\Curl\sym P \qquad \text{ in } \mathscr{D}'(\Omega,\R^{3\times3}).\] By virtue of the assumed regulartiy of the right hand side, we obtain that the left hand side thus belongs to  $W^{-1,\,p}(\Omega,\R^{3\times3})$ and moreover we have
\begin{align}\label{eq:curlskew}
\norm{\Curl \skew P}_{ W^{-1,\,p}(\Omega,\R^{3\times3})} &= \norm{\Curl(P-\sym P)}_{ W^{-1,\,p}(\Omega,\R^{3\times3})} \notag \\
&\le c \,(\norm{\Curl P}_{ W^{-1,\,p}(\Omega,\R^{3\times3})} + \norm{\sym P}_{ L^p(\Omega,\R^{3\times3})}),
\end{align}
hence $ \Curl \skew P \in  W^{-1,\,p}(\Omega,\R^{3\times3})$, so that  by Corollary \ref{cor:lin_combi} we deduce $\D \skew P\in  W^{-1,\,p}(\Omega,\R^{3\times3\times3})$.

Now, we can apply the above Lions lemma resp. Ne\v{c}as estimate (Theorem \ref{th:necas_estimate}) to $\skew P$. We arrive at $\skew P\in L^p(\Omega,\R^{3\times 3})$ and, moreover,
\begin{align}\label{eq:Norm_a}
 \norm{\skew P}_{L^p(\Omega,\R^{3\times 3})}&\le\quad  c\, (\norm{\skew P}_{ W^{-1,\,p}(\Omega,\R^{3\times3})} +  \norm{\D \skew P}_{ W^{-1,\,p}(\Omega,\R^{3\times3\times3})}) \notag \\
 &\overset{\mathclap{\text{Cor. \ref{cor:lin_combi}}}}{\leq}\quad c\, ( \norm{\skew P}_{ W^{-1,\,p}(\Omega,\R^{3\times 3})} +  \norm{\Curl \skew P}_{ W^{-1,\,p}(\Omega,\R^{3\times3})}) \notag \\
 &\overset{\mathclap{\eqref{eq:curlskew}}}{\leq}\quad c\, ( \norm{\skew P}_{ W^{-1,\,p}(\Omega,\R^{3\times3})} +  \norm{\Curl P}_{ W^{-1,\,p}(\Omega,\R^{3\times3})} + \norm{\sym P}_{L^p(\Omega,\R^{3\times3})}).
\end{align}
By adding $\norm{\sym P}_{L^p(\Omega,\R^{3\times3})}$ on both sides, the conclusion of the Lemma follows with regard to the orthogonal decomposition $P=\sym P + \skew P$.
\end{proof}

By eliminating the first term on the right-hand side of \eqref{eq:basic} we will arrive at our generalized Korn type inequalities, cf. Theorem \ref{thm:main1} and Theorem \ref{thm:main2}.

\begin{theorem}\label{thm:main1}
 Let $\Omega \subset \R^3$ be a bounded Lipschitz domain and $1<p<\infty$. There exists a constant $c=c(p,\Omega)>0$ such that for all $P\in  L^p(\Omega,\R^{3\times3})$,
 \begin{equation}\label{eq:Korn_Lp_w}
   \inf_{\widetilde{A}\in\so(3)}\norm{P-\widetilde{A}}_{L^p(\Omega,\R^{3\times3})}\leq c\,\left(\norm{ \sym P }_{L^p(\Omega,\R^{3\times3})}+ \norm{ \Curl P }_{W^{-1,\,p}(\Omega,\R^{3\times3})}\right).
 \end{equation}
\end{theorem}
\begin{proof}
We first start with the characterization of the kernel of the right-hand side,
$$
 K\coloneqq\{ P\in L^p(\Omega,\R^{3\times3}) \mid \sym P = 0  \text{ a.e. and } \Curl P = 0 \text{ in the distributional sense}\},
$$
by observing that $P\in K$ if and only if $P=\skew P$ and $P=\operatorname{const.}$ a.e. by virtue of Corollary \ref{cor:Nye}. Hence,
\begin{equation}\label{eq:charc_K}
 K=\{P\in L^p(\Omega,\R^{3\times3}) \mid  \exists\, \widetilde{A}\in\so(3) \text{ for which } P=\widetilde{A} \text{ a.e. in $\Omega$}\} \quad \text{and} \quad M \coloneqq\dim \so(3) =3.
\end{equation}
Consider the $M=3$ linear forms $\ell_{i+j-2}(P)\coloneqq\int_{\Omega}\skalarProd{e_i}{P\, e_j}\,\intd{x}=\int_\Omega P_{ij}\,\intd{x}$ on $L^p(\Omega,\R^{3\times3})$, $1\leq i < j\leq 3$. Then $P\in K$ is equal to $0$ a.e. if and only if $\ell_\alpha(P)=0$ for all $\alpha=1,2,3$. We claim that there exists a constant $c=c(p,\Omega)>0$, such that for all $P\in L^p(\Omega,\R^{3\times3})$ we have
\begin{equation}\label{eq:hilfsungl}
 \norm{P}_{L^p(\Omega,\R^{3\times3})}\leq c\,\left(\norm{ \sym P }_{L^p(\Omega,\R^{3\times3})}+ \norm{ \Curl P }_{W^{-1,\,p}(\Omega,\R^{3\times3})}+\sum_{\alpha=1}^3\abs{\ell_\alpha(P)} \right).
\end{equation}
Assume that \eqref{eq:hilfsungl} does not hold.\\
Then there would exist for each integer $k\in\N$ a function $P_k\in L^p(\Omega,\R^{3\times3})$ satisfying
\begin{equation}
 \norm{P_k}_{L^p(\Omega,\R^{3\times3})}=1 \quad \text{and}\quad \left(\norm{\sym P_k}_{L^p(\Omega,\R^{3\times3})}+\norm{\Curl P_k}_{ W^{-1,p}(\Omega,\R^{3\times3})}+ \sum_{\alpha=1}^3\abs{\ell_\alpha(P_k)}\right)< \frac1k.
\end{equation}
Thus, there exists a subsequence, which converges weakly to some $P^*$ in $L^p(\Omega,\R^{3\times3})$. It follows that $\sym P^*=0$ a.e., $\Curl P^*=0$ in the distributional sense and also $\ell_\alpha(P^*)=0$ for all $\alpha=1,2,3$. Thus, $P^*=0$ almost everywhere in $\Omega$.

Furthermore due to the compact embedding $L^p(\Omega,\R^{3\times3})\subset\!\subset  W^{-1,\,p}(\Omega,\R^{3\times3})$  there exists a subsequence, again denoted by $P_k$, so that $\skew P_k$ strongly converges to $0$ in $ W^{-1,\,p}(\Omega,\R^{3\times3})$. However, this conclusion is at variance with \eqref{eq:basic}:
\begin{equation}
 \begin{split}
 1 &= \norm{P_k}_{L^p(\Omega,\R^{3\times3})} \\
 & \overset{\eqref{eq:basic}}{\leq} c\, ( \norm{\skew P_k}_{ W^{-1,\,p}(\Omega,\R^{3\times3})} + \norm{\sym P_k}_{L^p(\Omega,\R^{3\times3})}+ \norm{\Curl P_k}_{ W^{-1,\,p}(\Omega,\R^{3\times3})} )\\
 &\qquad \to 0 \quad \text{for $k\to\infty$}.
 \end{split}
\end{equation}
This contradiction establishes estimate \eqref{eq:hilfsungl} for all $P\in L^p( \Omega,\R^{3\times3})$ and we arrive at the desired estimate:\\
For $P\in L^p(\Omega,\R^{3\times3})$, let $\widetilde{A}_P\in \so(3)$ be chosen in such a way that $\abs{\Omega}\left(\widetilde{A}_P\right)_{ij}=\int_\Omega P_{ij}\,\intd{x}$ for $1\leq i < j\leq 3$; in other words, $\ell_\alpha(P-\widetilde{A}_P)=0$ for all $\alpha=1,2,3$. Then
\begin{equation*}
 \inf_{\widetilde{A}\in \so(3)}\norm{P-\widetilde{A}}_{L^p(\Omega,\R^{3\times3})}\le \norm{P-\widetilde{A}_P}_{L^p(\Omega,\R^{3\times3})}  \overset{\eqref{eq:hilfsungl}}\leq c\,\left(\norm{ \sym P }_{L^p(\Omega,\R^{3\times3})}+ \norm{ \Curl P }_{W^{-1,\,p}(\Omega,\R^{3\times3})}\right).\qedhere
\end{equation*}
\end{proof}

\begin{remark}
 For compatible displacement gradients $P=\D  u$ we get back from \eqref{eq:Korn_Lp_w} the quantitative version of the classical Korn's inequality \eqref{eq:KornQuant} and for skew-symmetric matrix fields $P=A$ the corresponding Poincar\'{e} inequality since $\D  A = L(\Curl A)$.
\end{remark}

\begin{remark}
 To deduce the kernel of the right hand side of \eqref{eq:Korn_Lp_w}, we used Corollary \ref{cor:lin_combi} resp. \ref{cor:Nye}. Interestingly, on simply connected domains one can argue also in the following way:
 \begin{align*}
\Curl P \equiv 0 \ \Rightarrow \ P = \D  \vartheta,
 \end{align*}
 so, one  can apply the classical Korn's inequality \eqref{eq:KornQuant} to infer from $\sym \D  \vartheta \equiv 0$ that $\D \vartheta \equiv \operatorname{const} \in\so(3)$. 
 Finally, we examine the effect of homogeneous boundary conditions.
\end{remark}

\begin{theorem}\label{thm:main2}
Let $\Omega \subset \R^3$ be a bounded Lipschitz domain and $1<p<\infty$. There exists a constant $c=c(p,\Omega)>0$, such that for all $P\in  W^{1,\,p}_0(\Curl; \Omega,\R^{3\times3})$ we have
 \begin{equation}\label{eq:Korn_Lp_thm}
     \norm{ P }_{L^p(\Omega,\R^{3\times3})}\leq c\,\left(\norm{ \sym P }_{L^p(\Omega,\R^{3\times3})}+ \norm{ \Curl P }_{L^p(\Omega,\R^{3\times3})}\right).
 \end{equation}
\end{theorem}

\begin{proof}
We argue again by contradiction and assume that the estimate \eqref{eq:Korn_Lp_thm} does not hold. Then there exist functions $P_k\in W^{1,\,p}_0(\Curl; \Omega,\R^{3\times3})$ with the properties
$$
 \norm{P_k}_{L^p(\Omega,\R^{3\times3})}=1 \quad \text{and}\quad (\norm{\sym P_k}_{L^p(\Omega,\R^{3\times3})}+\norm{\Curl P_k}_{ L^p(\Omega,\R^{3\times3})})< \frac1k.
$$
Again $P_k\rightharpoonup P^*$ in $L^p(\Omega,\R^{3\times3})$ with $P^*\in K$. We now use the vanishing tangential trace condition to deduce that $P^*=0$ a.e.. In fact, since $\Curl P^* = 0$ a.e.,  we have for all $Q\in  W^{1-\frac{1}{p'},\,p'}(\partial\Omega,\R^{3\times 3})$
\begin{align}
 \dynabs{\skalarProd{P^*\times \nu}{Q}_{\partial\Omega}}= \lim_{k\to \infty} \dynabs{\int_\Omega\skalarProd{P_k}{\Curl \widetilde{Q}}\intd{x}}
 \le \lim_{k\to \infty} \norm{\Curl P_k}_{ L^p(\Omega,\R^{3\times3})}\,\norm{\widetilde{Q}}_{W^{1,\,p'}(\Omega,\R^{3\times3})}=0
\end{align}
where we have used $P_k\in W^{1,\,p}_0(\Curl;\Omega,\R^{3\times3})$ and $\widetilde{Q}\in W^{1,\,p'}(\Omega,\R^{3\times3})$ denotes any extension of $Q$ in $\Omega$. In other words, $P^*$ has vanishing tangential trace $P^*\times \nu = 0$. Moreover, since $P^*\in K$ by \eqref{eq:charc_K} there exists an $a^*\in \R^3$ such that
$$
P^*=\Anti (a^*)\in\so(3) \text{ a.e.}.
$$
Hence, the boundary condition $P^*\times \nu = \Anti(a^*)\times \nu= 0$ is also defined in the classical sense. By Observation~\ref{obs:alg_statement} we deduce $a^*=0$. We can now conclude as in the proof of Theorem \ref{thm:main1}
\end{proof}

\begin{remark}
 The same argumentation scheme applies to show that \eqref{eq:Korn_Lp_thm} also holds true for functions with vanishing tangential trace only on a relatively open (non-empty) subset $\Gamma\subseteq\partial\Omega$ of the boundary.
\end{remark}
\begin{remark}
It is well known, that Korn's inequality does not imply Poincar\'{e}'s inequality, however, due to the presence of the $\Curl$-part we get back both estimates from our generalization. Indeed,
for compatible $P=\D  u$ we recover from \eqref{eq:Korn_Lp_thm} a tangential Korn inequality and for skew-symmetric $P=A$ a Poincar\'{e} inequality.
\end{remark}

\begin{remark}
 The proof of Korn's inequality for the gradient of the displacement or for general incompatible tensor fields is mainly based on suitable representation formulas, cf. \eqref{eq:sec_der_id} and Corollary \ref{cor:lin_combi}, respectively. Cross-combining both conditions we obtain only infinitesimal rigid body motions:
 \begin{equation}
 \left.
 \begin{array}{llcl}
  \D  P = L(\D  \sym P) &\text{for }  P= \D  u &\ \& & P\in \so(3) \\[2ex]
  \D  A = L(\Curl A), &\text{for }  A\in\so(3) &\ \& & A= \D  u
 \end{array}
 \right\}\ \Longrightarrow\  u = \widetilde{A}\,x +\widetilde{b}, \ \text{with constant } \widetilde{A}\in\so(3), \widetilde{b}\in \R^3.
 \end{equation}
 See \cite{Smith70} for comparable representation formulas and deduced coercive inequalities.
 \end{remark}

\begin{remark}
It is  clear how to extend the present results to $(n\times n)$-square tensor fields with $n> 3$. The corresponding generalized $\Curl$ and tangential trace operation have already been presented in \cite{agn_munch2008curl,agn_neff2012maxwell}. This will be subject of a forthcoming note.
\end{remark}

\subsection{Open problems}
An interesting question is whether our result holds on domains more general than Lipschitz. For example the domain cannot have external cusps; indeed, both the classical Korn's inequality and the Lions lemma fail on such domains, cf. \cite{Weck1994,GG1998counterLions}. On the other hand, John domains support Korn-type inequalities, cf. \cite{DM2004Jonesdomains,DRS2010Johndomains,ADM2006Johndomains,Friedrich2018John}. John domains generalize the concept of Lipschitz domains, allowing certain fractal boundary structures but excluding the formation of external cusps.
However, there exist domains which are not John but
allow for Korn and Poincar\'{e} inequalities, see the discussions in \cite{DM2004Jonesdomains,BK1995John,JK2017John}.

A similar result concerning the  validity of the Lions lemma resp. Ne\v{c}as estimate would be interesting.

\subsubsection*{Acknowledgment}
This work does not have any conflicts of interest. The authors are grateful for inspiring discussions with Stefan M\"uller (Hausdorff Center for Mathematics, Bonn, Germany). This work was initiated in the framework of the Priority Programme SPP 2256 'Variational Methods for Predicting Complex Phenomena in Engineering Structures and Materials' funded by the Deutsche Forschungsgemeinschaft (DFG, German research foundation), Project-ID 422730790. The second author was supported within the project 'A variational scale-dependent transition scheme - from Cauchy elasticity to the relaxed micromorphic continuum’ (Project-ID 440935806). Moreover, both authors were supported in the Project-ID 415894848 by the Deutsche Forschungsgemeinschaft.

 \printbibliography

\end{document}